\documentclass[oneside,reqno,american]{amsart}
\usepackage[T1]{fontenc}
\usepackage[utf8]{inputenc}
\usepackage{xcolor}
\usepackage{babel}
\usepackage{prettyref}
\usepackage{amstext}
\usepackage{amsthm}
\usepackage{amssymb}
\usepackage{graphicx}
\usepackage[pdfusetitle,
 bookmarks=true,bookmarksnumbered=false,bookmarksopen=false,
 breaklinks=false,pdfborder={0 0 0},pdfborderstyle={},backref=false,colorlinks=false]
 {hyperref}
\hypersetup{
 colorlinks=true,citecolor=blue,linkcolor=blue,linktocpage=true}

\makeatletter

\providecommand{\tabularnewline}{\\}

\numberwithin{equation}{section}
\numberwithin{figure}{section}

\usepackage{prettyref}

\newrefformat{cor}{Corollary~\ref{#1}}
\newrefformat{subsec}{Section~\ref{#1}}
\newrefformat{lem}{Lemma~\ref{#1}}
\newrefformat{thm}{Theorem~\ref{#1}}
\newrefformat{sec}{Section~\ref{#1}}
\newrefformat{chap}{Chapter~\ref{#1}}
\newrefformat{prop}{Proposition~\ref{#1}}
\newrefformat{exa}{Example~\ref{#1}}
\newrefformat{tab}{Table~\ref{#1}}
\newrefformat{rem}{Remark~\ref{#1}}
\newrefformat{def}{Definition~\ref{#1}}
\newrefformat{fig}{Figure~\ref{#1}}
\newrefformat{claim}{Claim~\ref{#1}}

\makeatother

\theoremstyle{plain}
\newtheorem{thm}{\protect\theoremname}[section]
\newtheorem{lem}[thm]{\protect\lemmaname}
\theoremstyle{remark}
\newtheorem{rem}[thm]{\protect\remarkname}
\theoremstyle{plain}
\newtheorem{prop}[thm]{\protect\propositionname}
\newtheorem{cor}[thm]{\protect\corollaryname}
\providecommand{\corollaryname}{Corollary}
\providecommand{\lemmaname}{Lemma}
\providecommand{\propositionname}{Proposition}
\providecommand{\remarkname}{Remark}
\providecommand{\theoremname}{Theorem}

\begin{document}
\title{Shorting Dynamics and Structured Kernel Regularization}
\begin{abstract}
This paper develops a nonlinear operator dynamic that progressively
removes the influence of a prescribed feature subspace while retaining
maximal structure elsewhere. The induced sequence of positive operators
is monotone, admits an exact residual decomposition, and converges
to the classical shorted operator. Transporting this dynamic to reproducing
kernel Hilbert spaces yields a corresponding family of kernels that
converges to the largest kernel dominated by the original one and
annihilating the given subspace. In the finite-sample setting, the
associated Gram operators inherit a structured residual decomposition
that leads to a canonical form of kernel ridge regression and a principled
way to enforce nuisance invariance. This gives a unified operator-analytic
approach to invariant kernel construction and structured regularization
in data analysis.
\end{abstract}

\author{James Tian}
\address{Mathematical Reviews, 535 W. William St, Suite 210, Ann Arbor, MI
48103, USA}
\email{james.ftian@gmail.com}

\maketitle
\tableofcontents{}

\section{Introduction}\label{sec:1}

Many tasks in data analysis require removing or suppressing a known
nuisance component while preserving the structure of the remaining
signal. Examples include the extraction of feature subspaces, the
formation of invariant representations, and the regularization of
Gram matrices in kernel methods. These problems naturally lead to
operator-theoretic formulations in Hilbert space, where positive operators,
feature maps, and reproducing kernels encode the underlying geometry
of the data. The nonlinear shorting dynamic developed in this paper
provides a new way to analyze and implement such operations through
a monotone sequence of positive operators and their induced kernels.
The classical notion of shorted operators provides the structural
foundation for this dynamic.

Shorted operators were introduced and systematically studied by Anderson
and his collaborators in connection with matrix inequalities and electrical
network theory; see e.g., \cite{MR242573,MR287970,MR356949}. Their
construction provides, for a positive operator $R\in B(H)_{+}$ and
a closed subspace $U\subset H$, the maximal positive operator dominated
by $R$ whose range lies in $U^{\perp}$. The geometric and extension-theoretic
ideas underlying shorting is closely related to earlier work of Kreĭn
on selfadjoint extensions \cite{MR24574,MR24575}. These ideas have
since been expanded in several directions, including operator ranges
\cite{MR293441,MR531986,MR866966}, Schur complements and operator
inequalities \cite{MR1465881,MR2284176,MR2345997}, and spectral versions
of shorting \cite{MR2234254,MR2306006}.

In this paper, we revisit shorting from a dynamical viewpoint. Consider
the nonlinear update 
\[
R_{n+1}=R^{1/2}_{n}\left(I-T_{n+1}\right)R^{1/2}_{n},
\]
where each $T_{n}$ is an effect with $0\leq T_{n}\leq I$, and analyze
the dynamics generated by its repeated application. Its behavior turns
out to be surprisingly structured and connects directly to both classical
shorting theory and new variational, order-theoretic, and decomposition
properties. The induced map is nonlinear and order-preserving, and
produces a nontrivial trajectory $\left\{ R_{n}\right\} \subset B\left(H\right)_{+}$.
Our first goal is to analyze this trajectory. We show that the sequence
is monotone decreasing in the Löwner order, admits an additive residual
decomposition 
\[
R_{0}-R_{N}=\sum^{N-1}_{k=0}\left(R_{k}-R_{k+1}\right),
\]
and converges strongly to a positive limit $R_{\infty}$. When the
effects are supported in a closed subspace $U$ and exhaust it, this
limit reduces to the classical shorted operator of $R_{0}$ to $U^{\perp}$.

These results extend naturally to reproducing kernel Hilbert spaces.
If $K:S\times S\to\mathbb{R}$ is a positive definite kernel with
RKHS $\mathcal{H}_{K}$, then the interaction between kernel dominance
and operator structure can be traced to Aronszajn's classical work
\cite{MR51437}, and modern treatments \cite{MR3526117,MR2450103}.
If $U\subset\mathcal{H}_{K}$ is closed, and the effects $T_{n}$
act inside $U$, then the dynamics short the kernel to the orthogonal
complement of $U$. The resulting kernels $K_{n}$ form a Löwner-monotone
sequence converging to the maximal kernel dominated by $K$ whose
RKHS vanishes on $U$. 

A further aim is to describe the effect of this flow on kernel ridge
regression. The mathematical foundations of learning in RKHS date
back to the representer theorem and regularization framework of Cucker-Smale
\cite{MR1864085} and Engl-Hanke-Neubauer \cite{MR1408680,MR2146819}.
Given a finite sample and an initial Gram matrix $G_{0}$, the iterates
\[
G_{n+1}=G^{1/2}_{n}\left(I-T_{n+1}\right)G^{1/2}_{n}
\]
yield a sequence of Gram matrices on the same input points. For each
$n$, one may form the corresponding kernel ridge predictor $f_{n}$.
The sequence $\left\{ G_{n}\right\} $ admits an additive decomposition
of residuals that induces a canonical decomposition of the predictors
$f_{n}$. These structures are reminiscent of spectral filtering and
iterative schemes, but differ fundamentally in that the update is
geometric (via shorting) rather than spectral.

There is a natural interpretation in statistical learning. Many learning
problems involve a known nuisance factor or group attribute, such
as a protected demographic feature in fairness settings \cite{MR3388391,Kamiran:2012aa},
a domain or scanner label in multi-site data, or a spuriously correlated
covariate in robust prediction. In such settings one seeks predictors
that are invariant to this factor while retaining full expressive
power elsewhere. Scalar regularization cannot impose such invariance:
it shrinks all directions uniformly and does not distinguish nuisance
from signal. In contrast, the iterated shorting map enforces exact
invariance to $U$ at the level of the feature space and converges
to the maximal kernel satisfying this constraint. The induced decomposition
of $f_{n}$ into successive residuals provides a precise accounting
of how contributions aligned with $U$ are removed.

The contribution of this work is therefore threefold. First, we give
a self-contained analysis of the residual-weighted shorting flow in
$B\left(H\right)_{+}$, including monotonicity, convergence, and variational
structure, extending classical shorting theory into a dynamical regime
that appears not to have been examined previously. Second, we provide
a parallel theory for positive definite kernels, describing how shorting
interacts with RKHS structure and kernel dominance. Third, we show
how these constructions produce canonical invariant kernels and decompositions
for kernel ridge regression in the presence of a known nuisance subspace. 

The remainder of the paper is organized as follows. \prettyref{sec:2}
introduces the kernel shorting dynamics on the canonical feature space
$\mathcal{H}$. We construct the operator sequence $\left\{ R_{n}\right\} $
generated by the shorting updates, establish basic properties such
as positivity, Löwner monotonicity, and the residual identity, and
transport these results to the induced kernel sequence $\left\{ K_{n}\right\} $,
including existence of the limit kernel $K_{\infty}$.

\prettyref{sec:3} specializes to the case where the shorting acts
on a fixed closed subspace $U\subset\mathcal{H}$. Under a natural
exhaustion condition on $U$, we identify the limit $R_{\infty}$
with the projection $P_{U^{\perp}}$, obtain the explicit form of
the limiting kernel, and prove that $K_{\infty}$ is maximal among
kernels dominated by $K_{0}$ whose feature operators annihilate $U$.

\prettyref{sec:4} moves to the finite-sample setting. We study the
induced dynamics of Gram operators $G_{n}$, derive a telescoping
decomposition in terms of positive increments, and analyze the effect
on RKHS norms and kernel ridge regression, obtaining a dynamic decomposition
of the associated predictors.

\prettyref{sec:5} considers task-driven choices of the effects $T_{n}$.
We discuss several design principles (nuisance-space shorting, covariance-based
shorting, and greedy residual shorting), formulate an energy decomposition
with respect to a task operator $B$, and interpret the resulting
learning dynamics in terms of feature elimination and invariance.

\prettyref{sec:6} includes a geometric illustration of the shorting
dynamics through a two-dimensional example.

\section{Kernel Shorting Dynamics}\label{sec:2}

We start with the basic order and residual properties of the shorting
iterates and establish the existence of their limit, before turning
to the induced behavior of the associated kernels.

Let $H$ be a Hilbert space, and let $S$ be a nonempty set. Throughout,
all inner products are linear in the second variable.

A function $K:S\times S\rightarrow B\left(H\right)$ is called a $B\left(H\right)$-valued
positive definite kernel if for every finite selection $\left(s_{1},\dots,s_{m}\right)\in S$
and $\left(h_{1},\dots,h_{m}\right)\in H$, 
\[
\sum^{m}_{i,j=1}\left\langle h_{i},K\left(s_{i},s_{j}\right)h_{j}\right\rangle \ge0.
\]
There exists a Hilbert space $\mathcal{H}$ and a map $V:S\longrightarrow B\left(H,\mathcal{H}\right)$
such that 
\[
K\left(s,t\right)=V\left(s\right)^{*}V\left(t\right),
\]
and $\mathcal{H}$ is minimal in the sense that 
\[
\mathcal{H}=\overline{span}\left\{ V\left(s\right)h:s\in S,h\in H\right\} .
\]
This factorization is unique up to unitary equivalence. 

A canonical choice is to take $\mathcal{H}$ to be the reproducing
kernel Hilbert space (RKHS) $H_{\tilde{K}}$ of the associated scalar-valued
kernel defined on the product space $\left(S\times H\right)\times\left(S\times H\right)$
by
\[
\tilde{K}\left(\left(s,a\right),\left(t,b\right)\right)=\left\langle a,K\left(s,t\right)b\right\rangle 
\]
with $s,t\in S$ and $a,b\in H$, and set 
\[
V\left(t\right)a:=\tilde{K}\left(\cdot,\left(t,a\right)\right)
\]
as kernel sections. See e.g., \cite{MR2938971,MR4250453}.

All constructions below take place inside this canonical feature space
$\mathcal{H}=H_{\tilde{K}}$. We introduce a nonlinear transformation
on $\mathcal{H}$, which will serve as the fundamental dynamic governing
the evolution of new kernels derived from $K$.

Let $R_{0}=I_{\mathcal{H}}$, and choose any sequence of effects 
\[
T_{1},T_{2},\dots\in B\left(\mathcal{H}\right),\qquad0\le T_{n}\le I.
\]
Define recursively 
\[
R_{n+1}=R^{1/2}_{n}\left(I-T_{n+1}\right)R^{1/2}_{n},\qquad n\ge0.
\]
This is the shorting of $R_{n}$ along the effect $T_{n+1}$.

It induces a corresponding sequence of $B\left(H\right)$-valued kernels:
\begin{equation}
K_{n}\left(s,t\right)=V\left(s\right)^{*}R_{n}V\left(t\right),\qquad n\ge0\label{eq:b-0}
\end{equation}
where $K_{0}=K$. 

The purpose of this section is to establish the basic structural properties
of $\left(R_{n}\right)$ and $\left(K_{n}\right)$. The first lemma
records positivity and monotonicity of the operator sequence $\left(R_{n}\right)$.
\begin{lem}[Positivity and monotonicity]
For each $n\ge0$, $R_{n}$ is a positive operator on $\mathcal{H}$,
and 
\[
0\le R_{n+1}\le R_{n}\le R_{0}=I_{\mathcal{H}}.
\]
\end{lem}

\begin{proof}
Positivity is immediate since 
\[
R_{n+1}=R^{1/2}_{n}\left(I-T_{n+1}\right)R^{1/2}_{n}
\]
and both $R^{1/2}_{n}$ and $\left(I-T_{n+1}\right)$ are positive.

To verify $R_{n+1}\le R_{n}$, observe 
\[
R_{n+1}=R^{1/2}_{n}\left(I-T_{n+1}\right)R^{1/2}_{n}\le R^{1/2}_{n}IR^{1/2}_{n}=R_{n}.
\]
By induction, $R_{n}\le R_{0}=I$. 
\end{proof}
The next lemma gives a basic algebraic identity for everything that
follows. It shows that the amount removed at each step is given by
a positive operator of a special form.
\begin{lem}[Residual identity]
For every $n\ge0$, 
\begin{equation}
R_{n}-R_{n+1}=R^{1/2}_{n}T_{n+1}R^{1/2}_{n}.\label{eq:b-1}
\end{equation}
For every integer $N\ge1$, 
\begin{equation}
R_{0}-R_{N}=\sum^{N-1}_{m=0}R^{1/2}_{m}T_{m+1}R^{1/2}_{m}.\label{eq:b-2}
\end{equation}
\end{lem}

\begin{proof}
Expand the definition of $R_{n+1}$: 
\[
R_{n+1}=R^{1/2}_{n}\left(I-T_{n+1}\right)R^{1/2}_{n}=R^{1/2}_{n}R^{1/2}_{n}-R^{1/2}_{n}T_{n+1}R^{1/2}_{n}.
\]
Since $R_{n}=R^{1/2}_{n}R^{1/2}_{n}$, subtracting gives \eqref{eq:b-1}. 

Summing the residual identity over $m=0,\dots,N-1$ yields \eqref{eq:b-2}.
\end{proof}
We now identify the strong limit of the sequence $\left(R_{n}\right)$.
\begin{thm}
\label{thm:b-3}The sequence $\left(R_{n}\right)$ converges strongly
to a positive operator $R_{\infty}\in B\left(\mathcal{H}\right)$.
Moreover, 
\begin{equation}
R_{0}-R_{\infty}=\sum^{\infty}_{m=0}R^{1/2}_{m}T_{m+1}R^{1/2}_{m}\label{eq:b-3}
\end{equation}
with strong convergence of the series. 
\end{thm}

\begin{proof}
Since $\left(R_{n}\right)$ is a bounded monotone decreasing sequence
of positive operators, it converges pointwise on $\mathcal{H}$. For
each $x\in\mathcal{H}$, the sequence $\left\langle x,R_{n}x\right\rangle $
is decreasing and bounded below by $0$, hence convergent.

Define $R_{\infty}$ by 
\[
\left\langle x,R_{\infty}x\right\rangle =\lim_{n\to\infty}\left\langle x,R_{n}x\right\rangle .
\]
Standard arguments (see e.g., \cite{MR493419}) show that $R_{\infty}$
is positive and $R_{n}\to R_{\infty}$ strongly. Thus \eqref{eq:b-3}
follows from \eqref{eq:b-2} by letting $N\to\infty$. 
\end{proof}
We now transport these results through the feature map $V$ to obtain
the corresponding statements for the kernels $K_{n}$, see \eqref{eq:b-0}.
\begin{thm}
\label{thm:b-4}For each $n\ge0$, define 
\[
K_{n}\left(s,t\right)=V\left(s\right)^{*}R_{n}V\left(t\right).
\]
Then the following hold:
\begin{enumerate}
\item Each $K_{n}$ is a positive $B\left(H\right)$-valued kernel. 
\item The kernels decrease in the Löwner order: 
\[
K_{n+1}\left(s,t\right)\preceq K_{n}\left(s,t\right)\qquad\text{for all }s,t\in S.
\]
\item Defining 
\[
K^{(m)}\left(s,t\right)=V\left(s\right)^{*}R^{1/2}_{m}T_{m+1}R^{1/2}_{m}V\left(t\right),
\]
each $K^{(m)}$ is a positive $B\left(H\right)$-valued kernel and
\[
K_{0}\left(s,t\right)-K_{N}\left(s,t\right)=\sum^{N-1}_{m=0}K^{(m)}\left(s,t\right).
\]
\item The limit 
\[
K_{\infty}\left(s,t\right)=V\left(s\right)^{*}R_{\infty}V\left(t\right)
\]
exists in the strong operator topology, and 
\[
K_{0}\left(s,t\right)-K_{\infty}\left(s,t\right)=\sum^{\infty}_{m=0}K^{(m)}\left(s,t\right).
\]
\end{enumerate}
\end{thm}

\begin{proof}
(1) For each $n$, $R_{n}\ge0$, hence 
\begin{align*}
\sum^{m}_{i,j=1}\left\langle h_{i},K_{n}\left(s_{i},s_{j}\right)h_{j}\right\rangle  & =\sum^{m}_{i,j=1}\left\langle V\left(s_{i}\right)h_{i},R_{n}V\left(s_{j}\right)h_{j}\right\rangle \\
 & =\left\langle x,R_{n}x\right\rangle \ge0,
\end{align*}
with $x=\sum_{j}V\left(s_{j}\right)h_{j}$. Thus $K_{n}$ is positive
definite.

(2) Since $R_{n+1}\le R_{n}$, for all $s,t$, 
\begin{align*}
K_{n+1}\left(s,t\right) & =V\left(s\right)^{*}R_{n+1}V\left(t\right)\\
 & \preceq V\left(s\right)^{*}R_{n}V\left(t\right)=K_{n}\left(s,t\right).
\end{align*}

(3) Define $K^{(m)}$ as above. By positivity of each operator, $K^{(m)}$
is positive definite. Using \prettyref{eq:b-2}, 
\begin{align*}
K_{0}\left(s,t\right)-K_{N}\left(s,t\right) & =V\left(s\right)^{*}\left(R_{0}-R_{N}\right)V\left(t\right)\\
 & =\sum^{N-1}_{m=0}V\left(s\right)^{*}R^{1/2}_{m}T_{m+1}R^{1/2}_{m}V\left(t\right)=\sum^{N-1}_{m=0}K^{(m)}\left(s,t\right).
\end{align*}

(4) By \prettyref{thm:b-3}, $R_{n}\to R_{\infty}$ strongly. Thus
for each fixed $s,t$, 
\[
K_{n}\left(s,t\right)=V\left(s\right)^{*}R_{n}V\left(t\right)\longrightarrow V\left(s\right)^{*}R_{\infty}V\left(t\right)=K_{\infty}\left(s,t\right)
\]
strongly on $H$. Applying the telescoping identity in \prettyref{thm:b-3}
and transporting it through $V$ yields the infinite series decomposition. 
\end{proof}

\section{Invariance and Maximality of the Limit Kernel}\label{sec:3}

The kernel shorting dynamics introduced in the previous section produces
a decreasing family of kernels 
\[
K_{n}\left(s,t\right)=V\left(s\right)^{*}R_{n}V\left(t\right),\qquad n\ge0,
\]
with limit 
\[
K_{\infty}\left(s,t\right)=V\left(s\right)^{*}R_{\infty}V\left(t\right).
\]
In many applications one wants to suppress or remove a fixed ``nuisance''
subspace in the feature space $\mathcal{H}$. Typical examples include
subspaces generated by class means, low-frequency trends, or known
symmetries.

The goal is then to construct from $K_{0}$ a new kernel $\widehat{K}$
whose geometry is unchanged on the orthogonal complement of this nuisance
subspace, but which completely eliminates its contribution. On the
level of the feature space, the most direct way to remove a closed
subspace $U\subset\mathcal{H}$ is to project the features by $P_{U^{\perp}}$.
This produces the kernel 
\[
K_{U^{\perp}}\left(s,t\right)=V\left(s\right)^{*}P_{U^{\perp}}V\left(t\right),
\]
which ignores the components of $V\left(t\right)h$ that lie in $U$.

We will show that the kernel shorting dynamic provides a canonical
monotone path from $K_{0}$ down to $K_{U^{\perp}}$, and that the
limit kernel is maximal among all positive kernels dominated by $K_{0}$
whose feature operators annihilate $U$.

We fix a closed subspace $U\subset\mathcal{H}$ and denote by $P_{U}$
the orthogonal projection onto $U$, with $P_{U^{\perp}}=I-P_{U}$.
We assume that the shorting effects act only inside $U$ in the strong
sense 
\[
T_{n}=P_{U}T_{n}P_{U},\qquad n\ge1.
\]
Equivalently, each $T_{n}$ has block form 
\[
T_{n}=\begin{pmatrix}T_{n,U} & 0\\
0 & 0
\end{pmatrix}\quad\text{with respect to }\mathcal{H}=U\oplus U^{\perp}.
\]
The next lemma shows that, under this assumption, the shorting dynamic
leaves the orthogonal complement $U^{\perp}$ untouched.
\begin{lem}
\label{lem:3-1}Assume $T_{n}=P_{U}T_{n}P_{U}$ for all $n\ge1$.
Then for every $n\ge0$,
\begin{enumerate}
\item $U$ and $U^{\perp}$ reduce $R_{n}$; that is, 
\[
R_{n}=\begin{pmatrix}R_{n,U} & 0\\
0 & R_{n,U^{\perp}}
\end{pmatrix}\quad\text{on }U\oplus U^{\perp}.
\]
\item On $U^{\perp}$ we have $R_{n,U^{\perp}}=I_{U^{\perp}}$. In particular,
$R_{\infty}|_{U^{\perp}}=I_{U^{\perp}}$. 
\end{enumerate}
\end{lem}

\begin{proof}
We proceed by induction on $n$. For $n=0$ we have $R_{0}=I_{\mathcal{H}}$,
so both statements hold with $R_{0,U}=I_{U}$ and $R_{0,U^{\perp}}=I_{U^{\perp}}$.

Assume the statements hold for some $n$. Write $R_{n}$ in block
form relative to $U\oplus U^{\perp}$: 
\[
R_{n}=\begin{pmatrix}R_{n,U} & 0\\
0 & I_{U^{\perp}}
\end{pmatrix},
\]
and recall that 
\[
T_{n+1}=\begin{pmatrix}T_{n+1,U} & 0\\
0 & 0
\end{pmatrix}.
\]
Then 
\[
R^{1/2}_{n}=\begin{pmatrix}R^{1/2}_{n,U} & 0\\
0 & I_{U^{\perp}}
\end{pmatrix},\qquad\left(I-T_{n+1}\right)=\begin{pmatrix}I_{U}-T_{n+1,U} & 0\\
0 & I_{U^{\perp}}
\end{pmatrix}.
\]
Therefore 
\[
R_{n+1}=R^{1/2}_{n}\left(I-T_{n+1}\right)R^{1/2}_{n}=\begin{pmatrix}R^{1/2}_{n,U}\left(I_{U}-T_{n+1,U}\right)R^{1/2}_{n,U} & 0\\
0 & I_{U^{\perp}}
\end{pmatrix}.
\]
This shows that $U$ and $U^{\perp}$ reduce $R_{n+1}$, and that
$R_{n+1}|_{U^{\perp}}=I_{U^{\perp}}$. 

The strong limit $R_{\infty}$ inherits the same block form, with
$R_{\infty}|_{U^{\perp}}=I_{U^{\perp}}$. 
\end{proof}
The interesting behavior is therefore confined to the restriction
of $R_{n}$ to $U$. To obtain a clean invariance statement, we assume
that the shorting dynamic exhausts the subspace $U$: 
\begin{equation}
\lim_{n\to\infty}\left\langle x,R_{n}x\right\rangle =0\quad\text{for every }x\in U.\label{eq:c-1}
\end{equation}
This condition says that, in the limit, no nonzero vector in $U$
retains any ``energy'' under the sequence $\left(R_{n}\right)$.
Under this hypothesis the limit $R_{\infty}$ is exactly the projection
onto $U^{\perp}$.
\begin{lem}
\label{lem:3-2}Assume $T_{n}=P_{U}T_{n}P_{U}$ for all $n$ and that
\eqref{eq:c-1} holds. Then 
\[
R_{\infty}=P_{U^{\perp}}.
\]
\end{lem}

\begin{proof}
By \prettyref{lem:3-1}, $R_{\infty}$ has block diagonal form 
\[
R_{\infty}=\begin{pmatrix}R_{\infty,U} & 0\\
0 & I_{U^{\perp}}
\end{pmatrix}.
\]
For $x\in U$, \eqref{eq:c-1} gives 
\[
\left\langle x,R_{\infty,U}x\right\rangle =\lim_{n\to\infty}\left\langle x,R_{n,U}x\right\rangle =0.
\]
Positivity implies $R_{\infty,U}=0$. Therefore 
\[
R_{\infty}=\begin{pmatrix}0 & 0\\
0 & I_{U^{\perp}}
\end{pmatrix}=P_{U^{\perp}}.
\]
\end{proof}
We now translate this into a statement about the limiting kernel.
Define 
\[
K_{\infty}\left(s,t\right)=V\left(s\right)^{*}R_{\infty}V\left(t\right),
\]
as in \prettyref{thm:b-4}. Under the hypotheses above, we obtain
the explicit form 
\[
K_{\infty}\left(s,t\right)=V\left(s\right)^{*}P_{U^{\perp}}V\left(t\right),
\]
which may be read as ``project the feature map onto $U^{\perp}$
and then form the kernel.'' We next show that this limiting kernel
is maximal among kernels dominated by $K_{0}$ whose feature operators
annihilate $U$.

To make this precise we recall the standard correspondence between
kernels dominated by $K_{0}$ and positive contractions on $\mathcal{H}$.
\begin{lem}
\label{lem:rn}Let $K_{0}\left(s,t\right)=V\left(s\right)^{*}V\left(t\right)$
be the minimal factorization described in \prettyref{sec:2}. Let
$\widetilde{K}:S\times S\to B\left(H\right)$ be a positive kernel
such that $\widetilde{K}\preceq K_{0}$ (in the kernel Löwner order).
Then there exists a unique operator $Q\in B\left(\mathcal{H}\right)$
with 
\[
0\le Q\le I
\]
such that 
\[
\widetilde{K}\left(s,t\right)=V\left(s\right)^{*}QV\left(t\right)\quad\text{for all }s,t\in S.
\]
\end{lem}

\begin{proof}
[Sketch of proof] The assumption $\widetilde{K}\preceq K_{0}$ implies
that the sesquilinear form associated with $\widetilde{K}$ is bounded
by the form defined by $K_{0}$. This form extends to a bounded positive
operator $Q$ on $\mathcal{H}$ with $0\le Q\le I$ via the Riesz
representation theorem applied to the dense subspace $\mathrm{span}\left\{ V\left(s\right)h\right\} $.
The identity $\widetilde{K}\left(s,t\right)=V\left(s\right)^{*}QV\left(t\right)$
follows from a standard polarization argument. Uniqueness is immediate
from the minimality of $\mathcal{H}$. 
\end{proof}
We can now formulate the main maximality theorem.
\begin{thm}
[Kernel invariance and maximality] \label{thm:c-4}Let $U\subset\mathcal{H}$
be a closed subspace and assume $T_{n}=P_{U}T_{n}P_{U}$ for all $n\ge1$.
Assume further that the shorting dynamic exhausts $U$ in the sense
of \eqref{eq:c-1}. Let $K_{\infty}$ be the limiting kernel, so that
\[
K_{\infty}\left(s,t\right)=V\left(s\right)^{*}P_{U^{\perp}}V\left(t\right).
\]
Then:
\begin{enumerate}
\item $K_{\infty}$ is positive and $K_{\infty}\preceq K_{0}$. 
\item If $\widetilde{K}$ is any positive kernel with $\widetilde{K}\preceq K_{0}$
whose representing operator $Q$ from \prettyref{lem:rn} satisfies
\[
Q|_{U}=0,
\]
then 
\[
\widetilde{K}\left(s,t\right)\preceq K_{\infty}\left(s,t\right)\quad\text{for all }s,t\in S.
\]
\item Equivalently, $K_{\infty}$ is the unique maximal element (in the
Löwner order on kernels) among all positive kernels $\widetilde{K}\preceq K_{0}$
whose feature operators annihilate $U$. 
\end{enumerate}
\end{thm}

\begin{proof}
(1) Positivity and $K_{\infty}\preceq K_{0}$ follow directly from
\prettyref{thm:b-3}, since $0\le P_{U^{\perp}}\le I$ as an operator
on $\mathcal{H}$.

(2) Let $\widetilde{K}\preceq K_{0}$ and let $Q$ be the corresponding
positive contraction from \prettyref{lem:rn}. Relative to the decomposition
$\mathcal{H}=U\oplus U^{\perp}$, the condition $Q|_{U}=0$ implies
that $Q$ has block form 
\[
Q=\begin{pmatrix}0 & 0\\
0 & Q_{U^{\perp}}
\end{pmatrix},\qquad0\le Q_{U^{\perp}}\le I_{U^{\perp}}.
\]
On the other hand, 
\[
P_{U^{\perp}}=\begin{pmatrix}0 & 0\\
0 & I_{U^{\perp}}
\end{pmatrix}.
\]
Thus $Q\le P_{U^{\perp}}$ as operators on $\mathcal{H}$. Therefore,
for all $s,t\in S$ and $h\in H$, 
\begin{align*}
\left\langle h,\widetilde{K}\left(s,t\right)h\right\rangle  & =\left\langle V\left(s\right)h,QV\left(t\right)h\right\rangle \\
 & \le\left\langle V\left(s\right)h,P_{U^{\perp}}V\left(t\right)h\right\rangle \\
 & =\left\langle h,K_{\infty}\left(s,t\right)h\right\rangle .
\end{align*}
This shows $\widetilde{K}\left(s,t\right)\preceq K_{\infty}\left(s,t\right)$
for all $s,t$.

(3) If $\widetilde{K}$ is maximal among kernels with $\widetilde{K}\preceq K_{0}$
and $Q|_{U}=0$, then the inequality $\widetilde{K}\preceq K_{\infty}$
forces $\widetilde{K}=K_{\infty}$. 
\end{proof}
\begin{rem}
While the limiting kernel $K_{\infty}$ coincides with the static
projection $V\left(\cdot\right)^{*}P_{U^{\perp}}V\left(\cdot\right)$,
the key of the present construction lies in the dynamic itself. The
sequence $\left(K_{n}\right)$ constitutes a nonlinear monotone homotopy
from the original geometry $K_{0}$ to the invariant limit. Unlike
the static projection, the recursion provides an explicit residual
decomposition
\[
K_{0}-K_{N}=\sum^{N-1}_{m=0}K^{(m)},
\]
allowing for quantitative control over the suppression of the subspace
$U$. This extends the classical shorting theory of Anderson and Trapp
to the operator-valued kernel setting, characterizing the maximal
invariant kernel not merely as an algebraic projection, but as the
limit of a governed descent process in the Löwner order.

From a practical learning perspective, the intermediate kernels $K_{n}$
offer a mechanism for soft regularization that is unavailable in the
static setting. In many applications, such as fair representation
learning or domain adaptation, the ``nuisance'' subspace $U$ may
contain weak but task-relevant signals. Enforcing strict invariance
via $K_{\infty}$ can lead to excessive information loss. The shorting
dynamic allows the index $n$ to act as a discrete regularization
parameter, penalizing the subspace $U$ increasingly without strictly
forbidding it. This permits a tunable trade-off between invariance
and predictive utility, while the residual terms $K^{(m)}$ provide
an interpretable diagnostic of the geometry being filtered at each
stage.
\end{rem}

\section{Dynamic Kernel Regularization on Finite Samples}\label{sec:4}

In this section we study the evolution of the kernels 
\[
K_{n}\left(s,t\right)=V\left(s\right)^{*}R_{n}V\left(t\right),\qquad n\ge0,
\]
generated by the shorting dynamics on the canonical feature space
$\mathcal{H}$.

Our goal is to understand, at the level of finite samples and kernel-based
prediction, how the dynamic modifies Gram operators, restricts nuisance
components, and governs model complexity. Although the evolution takes
place in $\mathcal{H}$, all quantitative statements in this section
are kernel-native: they are expressed in terms of Gram operators on
a finite dataset, kernel sections, RKHS norms induced by $K_{n}$,
and the corresponding predictions. This ensures that the dynamic is
not merely a rephrasing of standard operator facts, but produces genuinely
new structure at the kernel and sample level.

\subsection{Gram Operators and Their Dynamic Evolution}

Fix a finite sample 
\[
X=\left\{ s_{1},\dots,s_{m}\right\} \subset S,
\]
and consider the product space $H^{m}=H\times\cdots\times H$. For
each kernel $K_{n}$, define the Gram operator 
\[
G_{n}:H^{m}\longrightarrow H^{m}
\]
by 
\[
\left(G_{n}h\right)_{i}:=\sum^{m}_{j=1}K_{n}\left(s_{i},s_{j}\right)h_{j},\qquad h=\left(h_{1},\dots,h_{m}\right)\in H^{m}.
\]
We define the sampling operator $V_{X}:H^{m}\to\mathcal{H}$ by 
\[
V_{X}\left(h_{1},\dots,h_{m}\right):=\sum^{m}_{i=1}V\left(s_{i}\right)h_{i}.
\]
The following properties are inherited from \prettyref{thm:b-3}.
\begin{lem}
\label{lem:4-1}For each $n\ge0$,
\begin{enumerate}
\item $G_{n}\ge0$; 
\item $G_{n+1}\le G_{n}$ in the Löwner order on $B\left(H^{m}\right)$; 
\item $G_{n}$ converges strongly to 
\[
G_{\infty}:=V^{*}_{X}R_{\infty}V_{X};
\]
\item the telescoping decomposition holds: 
\[
G_{0}-G_{N}=\sum^{N-1}_{m=0}\Delta G^{(m)},
\]
where the increments are positive operators given by 
\[
\Delta G^{(m)}=V^{*}_{X}R^{1/2}_{m}T_{m+1}R^{1/2}_{m}V_{X}\ge0.
\]
\end{enumerate}
\end{lem}

\begin{proof}
Every assertion follows by applying the map $Q\mapsto V^{*}_{X}QV_{X}$
to the corresponding operator identities for $R_{n}$ and using the
positivity of the map $Q\mapsto V^{*}_{X}QV_{X}$. For example, 
\[
G_{n+1}=V^{*}_{X}R_{n+1}V_{X}\le V^{*}_{X}R_{n}V_{X}=G_{n}.
\]
The decomposition follows from $R_{0}-R_{N}=\sum^{N-1}_{m=0}R^{1/2}_{m}T_{m+1}R^{1/2}_{m}$
and positivity is clear. 
\end{proof}
The operators $\Delta G^{(m)}$ represent the ``portion of the Gram
operator removed at step $m+1$.'' This decomposition will be essential
in studying prediction and regularization.

\subsection{RKHS Norms and Dynamic Geometry}

Let $\mathcal{H}_{K_{n}}$ denote the RKHS of the kernel $K_{n}$.
Given the minimal feature map $V$, each function $f\in\mathcal{H}_{K_{n}}$
admits a representation 
\[
f\left(t\right)=V\left(t\right)^{*}u,\qquad u\in\mathcal{H}.
\]
The RKHS norm is given by 
\[
\left\Vert f\right\Vert ^{2}_{K_{n}}=\left\langle u,R^{-1}_{n}u\right\rangle _{\mathcal{H}}
\]
whenever $u\in\overline{ran}(R^{1/2}_{n})$, and is infinite otherwise.
This expression shows that the dynamic modifies the geometry of the
hypothesis space: directions removed by $R_{n}$ become penalized
infinitely, and therefore disappear from $\mathcal{H}_{K_{n}}$. We
state this in a form adapted to samples.
\begin{prop}
\label{prop:4-2}For any fixed feature vector $u\in\mathcal{H}$,
the function $f(\cdot)=V(\cdot)^{*}u$ satisfies 
\[
R_{n+1}\le R_{n}\implies\left\Vert f\right\Vert _{K_{n+1}}\ge\left\Vert f\right\Vert _{K_{n}}.
\]
In other words, the cost of representing a fixed function increases
as $n$ increases.
\end{prop}

\begin{proof}
The representation $f=V(\cdot)^{*}u$ identifies the feature vector
associated with $f$ in the geometry of $K_{n}$ as $u_{n}=R^{-1/2}_{n}u$
(since $K_{n}(s,t)=V(s)^{*}R^{1/2}_{n}R^{1/2}_{n}V(t)$). Thus 
\[
\left\Vert f\right\Vert ^{2}_{K_{n}}=\left\langle u_{n},u_{n}\right\rangle =\langle R^{-1/2}_{n}u,R^{-1/2}_{n}u\rangle=\left\langle u,R^{-1}_{n}u\right\rangle .
\]
Monotonicity $R_{n+1}\le R_{n}$ implies $R^{-1}_{n+1}\ge R^{-1}_{n}$
on their respective ranges (Löwner monotonicity of the inverse function),
establishing the inequality.
\end{proof}
Thus the dynamic ``tightens'' the RKHS geometry: as $n$ grows,
maintaining the same function output requires strictly more ``energy''
(norm), effectively regularizing the solution away from the shorted
directions.

\subsection{Dynamic Kernel Ridge Regression on Finite Samples}

We now examine Kernel Ridge Regression (KRR) with the dynamically
evolving kernel family $\left(K_{n}\right)$. Given data $\left(X,y\right)$
with $y=\left(y_{1},\dots,y_{m}\right)\in H^{m}$ and regularization
parameter $\lambda>0$, define the KRR solution using $K_{n}$: 
\[
f_{n}:=\arg\min_{f\in\mathcal{H}_{K_{n}}}\left(\sum^{m}_{i=1}\left\Vert f\left(s_{i}\right)-y_{i}\right\Vert ^{2}+\lambda\left\Vert f\right\Vert ^{2}_{K_{n}}\right).
\]
Using the representer theorem in the vector-valued RKHS setting, we
have 
\[
f_{n}\left(t\right)=\sum^{m}_{i=1}K_{n}\left(t,s_{i}\right)c_{n,i},\qquad c_{n}\in H^{m},
\]
where 
\[
\left(G_{n}+\lambda I_{H^{m}}\right)c_{n}=y.
\]
The Gram dynamics therefore induce a dynamic regularization path 
\[
c_{0},c_{1},c_{2},\dots,c_{\infty},\qquad f_{0},f_{1},f_{2},\dots,f_{\infty}.
\]
The next theorem shows that this path decomposes into an invariant
part plus a decaying nuisance part, with explicit control using the
residual Gram operators.
\begin{thm}
\label{thm:4-3}For each $n\ge0$, let $c_{n}$ satisfy 
\[
\left(G_{n}+\lambda I\right)c_{n}=y.
\]
Then:
\begin{enumerate}
\item The coefficients satisfy 
\[
c_{N}-c_{0}=\sum^{N-1}_{m=0}\left(G_{m+1}+\lambda I\right)^{-1}\Delta G^{(m)}c_{m},
\]
with all terms nonnegative in the sense of the inner product of $H^{m}$. 
\item If the shorting dynamic exhausts a closed subspace $U\subset\mathcal{H}$
and $T_{n}$ are supported in $U$, then for any $g$ whose feature
vector lies in $U$, 
\[
\left\langle f_{n},g\right\rangle _{K_{0}}\longrightarrow0,\qquad\text{as }n\to\infty.
\]
Thus the nuisance component of the predictor is removed dynamically. 
\end{enumerate}
\end{thm}

\begin{proof}
(1) We use the resolvent identity 
\[
A^{-1}-B^{-1}=B^{-1}\left(B-A\right)A^{-1}.
\]
Set $A=G_{m}+\lambda I$ and $B=G_{m+1}+\lambda I$. Since 
\[
G_{m}-G_{m+1}=\Delta G^{(m)},
\]
we have $B-A=-\Delta G^{(m)}$. Thus 
\[
\left(G_{m}+\lambda I\right)^{-1}-\left(G_{m+1}+\lambda I\right)^{-1}=\left(G_{m+1}+\lambda I\right)^{-1}\Delta G^{(m)}\left(G_{m}+\lambda I\right)^{-1}.
\]
Apply both sides to $y$, noting that $c_{m}=\left(G_{m}+\lambda I\right)^{-1}y$,
to obtain 
\[
c_{m}-c_{m+1}=\left(G_{m+1}+\lambda I\right)^{-1}\Delta G^{(m)}c_{m}.
\]
Rearranging to $c_{m+1}-c_{m}=\dots$ and summing from $m=0$ to $N-1$
yields the result.

(2) If the shorting exhausts $U$, then $V\left(s_{i}\right)c_{m,i}$
retains no $U$-component in the limit. For any $g$ with feature
vector in $U$, $\left\langle f_{n},g\right\rangle \to0$ follows
from orthogonality of feature vectors and the representation of $f_{n}$
in the feature space via $R_{n}$. 
\end{proof}

\subsection{Interpretation}

The dynamic kernel path $K_{0}\succeq K_{1}\succeq\cdots\succeq K_{\infty}$
provides a new form of kernel regularization guided by operator shorting
dynamics rather than scalar tuning. The path is monotone (each step
removes positive kernel content), structured (the removed content
is exactly $K^{(m)}$), and interpretable (each $K^{(m)}$ represents
the geometry filtered out at step $m$). Crucially, if the dynamic
targets a subspace $U$, the $U$-component of the predictor decays
to zero, while the solution evolves predictably via residual Gram
components.

\section{Task-Driven Shorting and Learning Dynamics}\label{sec:5}

Up to this point the effects $T_{n}$ have been arbitrary, subject
only to $0\le T_{n}\le I$, and the analysis has been purely structural.
We now turn to the learning setting, where the choice of $T_{n}$
is driven by a task. The central idea is to use shorting as a mechanism
for dynamically eliminating feature directions in $\mathcal{H}$ that
are deemed nuisance or task-irrelevant, while retaining those necessary
for prediction.

Throughout this section we fix a finite sample $X=\{s_{1},\dots,s_{m}\}\subset S$
and labels $y=(y_{1},\dots,y_{m})\in H^{m}$, and we work with the
Gram operators $G_{n}$ and residuals $\Delta G^{(m)}$ defined in
\prettyref{sec:4}. Results in this section are expressed at the level
of $G_{n}$, the KRR solutions $c_{n},f_{n}$, and task-driven choices
of $T_{n}$.

\subsection{Task-Driven Design of the Effects}

We first describe a few natural design principles for the effects
$T_{n}$. In each case, the shorting update 
\[
R_{n+1}=R^{1/2}_{n}(I-T_{n+1})R^{1/2}_{n}
\]
remains the core engine, with the task entering only through the choice
of $T_{n+1}$.

\subsubsection*{Nuisance-space shorting}

Suppose prior knowledge identifies a closed subspace $U\subset\mathcal{H}$
that carries nuisance variation, such as a known family of confounders
(correlated nuisance variables), style directions (symmetry generators),
or uninformative covariates (high-variance noise). The simplest choice
is 
\[
T_{n}=P_{U},\qquad n\ge1.
\]
As shown in \prettyref{sec:2}, this forces $R_{n}\to P_{U^{\perp}}$
strongly, so that the limiting kernel is 
\[
K_{\infty}(s,t)=V(s)^{*}P_{U^{\perp}}V(t),
\]
the shorted kernel that annihilates $U$ and is maximal under the
constraint $K_{\infty}\preceq K_{0}$. \prettyref{sec:4} already
shows that, at the KRR level, the component of the predictor with
feature vector in $U$ is removed dynamically. We revisit this in
a task-focused form below.

\subsubsection*{Data-driven covariance shorting}

Let 
\[
\Sigma_{X}:=\frac{1}{m}\sum^{m}_{i=1}V(s_{i})V(s_{i})^{*}
\]
be the empirical feature covariance operator on $\mathcal{H}$ associated
with the sample $X$. One may choose $T_{n+1}$ as a spectral projector
of $\Sigma_{X}$, 
\[
T_{n+1}=\chi_{[\tau_{n+1},\infty)}(\Sigma_{X}),
\]
where the threshold $\tau_{n+1}\ge0$ is tuned on a validation set.
Shorting along $T_{n+1}$ then removes high-variance feature directions
(e.g., dominant nuisance modes) in a nonlinear, operator-theoretic
way that cannot be reproduced by scalar shrinkage on $G_{n}$ alone.

\subsubsection*{Greedy residual shorting}

More adaptively, one can choose each $T_{n+1}$ to target the direction
in $\mathcal{H}$ that makes the largest contribution to a given quadratic
task functional. A basic instance, analyzed in detail below, is obtained
by fixing a positive operator $B\in B(\mathcal{H})$ (encoding the
task) and choosing $T_{n+1}$ as the rank-one projection onto the
leading eigenvector of 
\[
B_{n}:=R^{1/2}_{n}BR^{1/2}_{n}.
\]
In this case, each shorting step removes the maximal possible amount
of task energy $\langle B,R_{n}\rangle$ among all rank-one shortings
of $R_{n}$. This produces a genuinely new greedy feature-elimination
dynamic.

\subsection{Shorting-Based Energy Decomposition}

We now make the preceding discussion precise. Fix a positive operator
$B\in B(\mathcal{H})$, which we interpret as a task operator. In
applications, $B$ may be chosen as an empirical covariance, or as
an operator encoding a linear functional of interest. For each $n$,
define the \emph{task energy} of the current residual operator $R_{n}$
by 
\[
E_{n}:=\left\langle B,R_{n}\right\rangle :=\mathrm{tr}(R^{1/2}_{n}BR^{1/2}_{n}),
\]
whenever $BR_{0}$ is trace class so that $E_{0}<\infty$. The shorting
update induces a simple identity for the energy drop.
\begin{lem}
\label{lem:5-1}For each $n\ge0$, 
\[
E_{n}-E_{n+1}=\langle R^{1/2}_{n}BR^{1/2}_{n},T_{n+1}\rangle,
\]
where the inner product on the right-hand side is the Hilbert-Schmidt
inner product whenever $R^{1/2}_{n}BR^{1/2}_{n}$ is trace class.
In particular, 
\[
E_{n}-E_{n+1}\ge0,
\]
and the per-step energy drop is governed by $T_{n+1}$ and the conjugated
task operator $B_{n}=R^{1/2}_{n}BR^{1/2}_{n}$. 
\end{lem}

\begin{proof}
Using the shorting update, 
\[
R_{n+1}=R^{1/2}_{n}(I-T_{n+1})R^{1/2}_{n}=R_{n}-R^{1/2}_{n}T_{n+1}R^{1/2}_{n},
\]
we obtain 
\begin{align*}
E_{n}-E_{n+1} & =\left\langle B,R_{n}-R_{n+1}\right\rangle =\langle B,R^{1/2}_{n}T_{n+1}R^{1/2}_{n}\rangle.
\end{align*}
Positivity of $R^{1/2}_{n}T_{n+1}R^{1/2}_{n}$ and $B$ implies the
right-hand side is nonnegative. Whenever $BR_{0}$ is trace class,
the Hilbert-Schmidt inner product is well defined and 
\[
\langle B,R^{1/2}_{n}T_{n+1}R^{1/2}_{n}\rangle=\mathrm{tr}\big(R^{1/2}_{n}BR^{1/2}_{n}T_{n+1}\big)=\langle R^{1/2}_{n}BR^{1/2}_{n},T_{n+1}\rangle,
\]
which yields the stated identity. 
\end{proof}
Thus shorting can be seen as a mechanism that decreases $E_{n}$ by
an amount determined by the overlap of $T_{n+1}$ with $B_{n}=R^{1/2}_{n}BR^{1/2}_{n}$.

\subsection{Greedy Rank-One Shorting}

We now show that in the finite-rank, rank-one case shorting induces
a natural greedy strategy: at each step, choose the effect that maximizes
task energy removal.

Assume for this subsection that $R_{n}$ has finite rank for all $n$
and that we restrict $T_{n+1}$ to be a rank-one projection supported
in $\overline{ran}(R_{n})$. Thus 
\[
T_{n+1}=\left|w_{n}\left\rangle \right\langle w_{n}\right|
\]
for some unit vector $w_{n}\in\overline{ran}(R_{n})$. Then \prettyref{lem:5-1}
becomes 
\[
E_{n}-E_{n+1}=\langle R^{1/2}_{n}BR^{1/2}_{n},\left|w_{n}\left\rangle \right\langle w_{n}\right|\rangle=\langle w_{n},R^{1/2}_{n}BR^{1/2}_{n}w_{n}\rangle.
\]

\begin{thm}
\label{thm:greedy}Suppose $R_{n}$ has finite rank and we restrict
to rank-one effects $T_{n+1}=\left|w\left\rangle \right\langle w\right|$
with $\left\Vert w\right\Vert =1$ and $w\in\overline{ran}(R_{n})$.
Then the shorting step that maximizes the energy drop $E_{n}-E_{n+1}$
is obtained by choosing $w_{n}$ to be an eigenvector corresponding
to the largest eigenvalue of $B_{n}:=R^{1/2}_{n}BR^{1/2}_{n}$. The
maximum energy drop is given by 
\[
E_{n}-E_{n+1}=\lambda_{\max}(B_{n}),
\]
and no other rank-one effect achieves a larger drop. 
\end{thm}

\begin{proof}
Under the stated assumptions, 
\[
E_{n}-E_{n+1}=\left\langle w,B_{n}w\right\rangle ,
\]
where $B_{n}$ is positive and finite rank. The Rayleigh quotient
variational principle for positive operators asserts that 
\[
\lambda_{\max}(B_{n})=\max_{\left\Vert w\right\Vert =1}\left\langle w,B_{n}w\right\rangle ,
\]
with the maximum attained precisely on the eigenspace associated with
$\lambda_{\max}(B_{n})$. Note that since $B_{n}$ vanishes on $\ker\left(R_{n}\right)$,
any eigenvector corresponding to a non-zero eigenvalue automatically
lies in $\overline{\mathrm{ran}}\left(R_{n}\right)$. Thus any unit
eigenvector $w_{n}$ corresponding to $\lambda_{\max}(B_{n})$ yields
\[
E_{n}-E_{n+1}=\left\langle w_{n},B_{n}w_{n}\right\rangle =\lambda_{\max}(B_{n}),
\]
and no other unit vector $w$ produces a larger value. 
\end{proof}
Thus, within the class of rank-one shortings, the choice of $T_{n+1}$
that removes the largest amount of task energy $E_{n}$ is precisely
the shorting along the leading eigenvector of the conjugated task
operator $R^{1/2}_{n}BR^{1/2}_{n}$. This greedy principle is genuinely
nonlinear in $R_{n}$ and is a form of metric-driven feature elimination
that has no direct equivalent in standard scalar kernel regularization.

\subsection{Specialization to Kernel Ridge Regression}

We now specialize these ideas to kernel ridge regression. Recall that
for each $n$, the KRR solution with kernel $K_{n}$ and regularization
parameter $\lambda>0$ is 
\[
f_{n}(t)=\sum^{m}_{i=1}K_{n}(t,s_{i})c_{n,i},\qquad(G_{n}+\lambda I)c_{n}=y.
\]
\prettyref{thm:4-3} shows that the KRR path $\{f_{n}\}$ admits an
explicit decomposition in terms of the residual kernels $K^{(m)}$
and Gram increments $\Delta G^{(m)}$. We now connect this to a task-driven
choice of $B$.

A natural choice is to fix $B$ as the empirical feature covariance
associated with the training inputs: 
\[
B_{X}:=\frac{1}{m}\sum^{m}_{i=1}V(s_{i})V(s_{i})^{*}.
\]
Choosing $T_{n+1}$ according to \prettyref{thm:greedy} with this
$B_{X}$ produces a feature-suppression dynamic that greedily removes
the directions in $\mathcal{H}$ that account for the largest portion
of the empirical feature variance.

Combined with the results in \prettyref{thm:4-3}, this implies that
the KRR predictor adapts by shedding support on the feature directions
selected by $T_{m+1}$. In particular, the contribution to $f_{0}$
from the top eigen-directions of $B_{0},B_{1},\dots,B_{N-1}$ is removed
after $N$ steps.

\subsection{Nuisance-Invariant Limits}

We conclude by returning to the nuisance-space shorting scheme, which
specializes the abstract dynamic invariance of \prettyref{sec:2}
to the present learning setting. For clarity, we state the result
as a corollary of \prettyref{thm:4-3}.
\begin{cor}
\label{cor:5-3}Let $U\subset\mathcal{H}$ be a closed subspace and
set $T_{n}=P_{U}$ for all $n\ge1$. Let $f_{n}$ be the KRR solutions
with kernels $K_{n}$. Then:
\begin{enumerate}
\item The limiting kernel is $K_{\infty}(s,t)=V(s)^{*}P_{U^{\perp}}V(t)$,
and the limiting KRR predictor $f_{\infty}$ depends only on the component
of the labels that is visible through $U^{\perp}$. 
\item For any $g\in\mathcal{H}_{K_{0}}$ whose feature vector lies in $U$,
\[
\left\langle f_{n},g\right\rangle _{K_{0}}\longrightarrow0.
\]
In particular, the nuisance component of $f_{n}$ with feature in
$U$ is removed dynamically. 
\item Among all kernels $K\preceq K_{0}$ whose feature map annihilates
$U$, the kernel $K_{\infty}$ is maximal, and therefore $f_{\infty}$
is the least regularized KRR solution subject to full nuisance removal. 
\end{enumerate}
\end{cor}

\begin{proof}
(1) and (2) follow directly from \prettyref{thm:4-3} part (3), and
the fact that $R_{n}\to P_{U^{\perp}}$ strongly when $T_{n}=P_{U}$. 

For (3), let $\widetilde{K}\preceq K_{0}$ be any kernel whose canonical
feature map $\widetilde{V}$ satisfies $\widetilde{V}(s)h\in U^{\perp}$
for all $(s,h)$. Then the associated operator $\widetilde{R}$ on
$\mathcal{H}$ satisfies $0\le\widetilde{R}\le I$ and $\widetilde{R}u=0$
for all $u\in U$. Since $P_{U^{\perp}}$ is the orthogonal projection
onto $U^{\perp}$, we have $\widetilde{R}\le P_{U^{\perp}}$, hence
$\widetilde{K}(s,t)\preceq K_{\infty}(s,t)$ for all $s,t$. This
shows maximality of $K_{\infty}$ and, consequently, that the KRR
solution with $K_{\infty}$ is the least regularized solution among
all kernels achieving full nuisance removal. 
\end{proof}
This corollary shows a form of task-adapted invariance: once a nuisance
subspace $U$ is identified at the feature level, the shorting dynamic
provides a canonical path of kernels and KRR predictors that converges
to the maximal nuisance-free kernel $K_{\infty}$ and the corresponding
predictor $f_{\infty}$. This behavior cannot be replicated by scalar
kernel tuning alone; it crucially exploits the operator shorting dynamics
on $\mathcal{H}$.

\subsection{Summary}

The results in this section show that shorting can be used as a genuinely
new mechanism for dynamic kernel regularization driven by the feature
space geometry. The shorting map acts as the engine that 
\begin{itemize}
\item generates a decreasing sequence of kernels $K_{n}$ with positive
residual kernels $K^{(m)}$, 
\item induces a positive decomposition of the Gram operators $G_{n}$ and
a structured additive decomposition of the KRR predictors $f_{n}$, 
\item supports greedy, task-aligned feature elimination via \prettyref{thm:greedy},
and 
\item yields nuisance-invariant predictors via \prettyref{cor:5-3}. 
\end{itemize}

\section{Illustration of the Shorting Dynamics}\label{sec:6}

The results in Sections \ref{sec:2}--\ref{sec:3} describe how the
shorting iteration 
\[
R_{n+1}=R^{1/2}_{n}\left(I-T\right)R^{1/2}_{n}
\]
removes the contribution of a closed subspace $U\subset H$ while
retaining the maximal possible structure on $U^{\perp}$ under one-sided
domination. To provide a direct geometric interpretation of this effect,
we include a finite-dimensional example in which the entire dynamic
is visible in the plane. The illustration makes the behavior of the
kernels $\left\{ K_{n}\right\} $ and their associated estimators
visually transparent.

We take $H=\mathbb{R}^{2}$ equipped with its usual inner product
and coordinates $\left(s,u\right)$. The coordinate $s$ is regarded
as the signal direction, and the coordinate $u$ plays the role of
a nuisance direction. We consider the nuisance subspace 
\[
U=\left\{ \left(0,u\right):u\in\mathbb{R}\right\} ,\qquad U^{\perp}=\left\{ \left(s,0\right):s\in\mathbb{R}\right\} ,
\]
so that the orthogonal projection onto $U^{\perp}$ is the matrix
$\left[\begin{smallmatrix}1 & 0\\
0 & 0
\end{smallmatrix}\right]$. To mimic the shorting iteration, we introduce a monotone family
of positive operators 
\[
R_{n}=\left(\begin{matrix}1 & 0\\
0 & \left(1-c\right)^{n}
\end{matrix}\right),\qquad0<c<1,
\]
which leaves the signal direction unchanged and contracts the nuisance
direction by a factor $\left(1-c\right)^{n}$. In particular, $R_{n+1}\le R_{n}$
and $R_{n}\to P_{U^{\perp}}$ in the strong topology.

Each operator $R_{n}$ induces a kernel 
\[
K_{n}\left(\left(s,u\right),\left(s',u'\right)\right)=\langle R^{1/2}_{n}\left(s,u\right),R^{1/2}_{n}\left(s',u'\right)\rangle=ss'+\left(1-c\right)^{n}uu'.
\]
Thus $K_{0}$ treats the two coordinates symmetrically, whereas $K_{n}$
increasingly suppresses the nuisance coordinate, and 
\[
K_{\infty}\left(s,u;s',u'\right)=\lim_{n\to\infty}K_{n}\left(s,u;s',u'\right)=ss'
\]
is precisely the kernel corresponding to the shorted operator $P_{U^{\perp}}$.

To visualize the effect of $\left\{ K_{n}\right\} $, we generate
a finite set of points $\left(s_{i},u_{i}\right)$ in the square $\left[-1,1\right]^{2}\subset\mathbb{R}^{2}$
and assign labels by a tilted rule 
\[
y_{i}=\mathrm{sign}\left(s_{i}+\alpha u_{i}\right),\qquad\alpha>0.
\]
The separating boundary associated with these labels is not vertical;
it depends on both coordinates. For each $n$, we compute the kernel
ridge estimator $f_{n}$ associated with $K_{n}$. Since all kernels
in this example are linear in $\left(s,u\right)$, each estimator
is an affine function 
\[
f_{n}\left(s,u\right)=a_{n}s+b_{n}u+d_{n},
\]
and its decision boundary is a straight line in the plane.

\prettyref{fig:sd} shows the boundaries $f_{n}\left(s,u\right)=0$
for three intermediate values $n\in\left\{ 0,2,6\right\} $ and for
the shorted limit $n=\infty$. The sequence illustrates the following
behavior, which mirrors the abstract operator dynamics: 
\begin{itemize}
\item At $n=0$, the kernel $K_{0}\left(s,u;s',u'\right)=ss'+uu'$ permits
full dependence on the nuisance coordinate. The estimator aligns with
the tilted separating rule, and the resulting boundary is oblique. 
\item At moderate values of $n$, the factor $\left(1-c\right)^{n}$ significantly
suppresses the contribution of the second coordinate. The boundary
rotates toward vertical, reflecting the decreasing sensitivity of
the estimator to the nuisance direction. 
\item As $n\to\infty$, the operators $R_{n}$ converge to $P_{U^{\perp}}$
and the kernels converge to $K_{\infty}\left(s,u;s',u'\right)=ss'$.
The limiting estimator depends only on the signal coordinate $s$,
and the boundary becomes vertical, coinciding with the invariant separating
rule $s=0$. 
\end{itemize}
\begin{figure}[ht]
\begin{tabular}{cc}
\includegraphics[width=0.45\columnwidth]{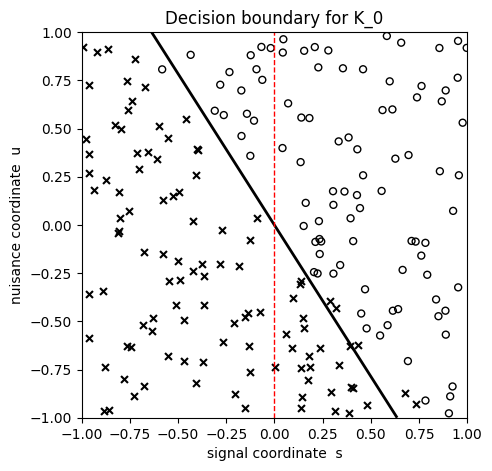} & \includegraphics[width=0.45\columnwidth]{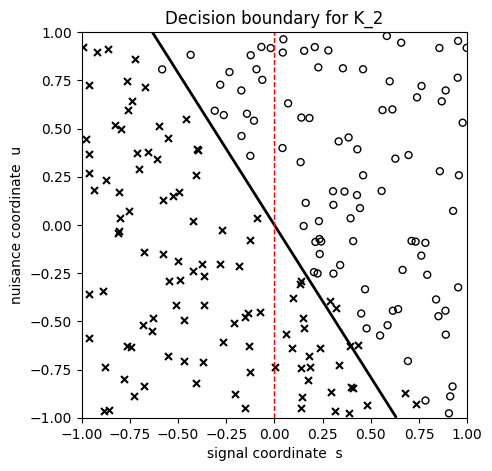}\tabularnewline
\includegraphics[width=0.45\columnwidth]{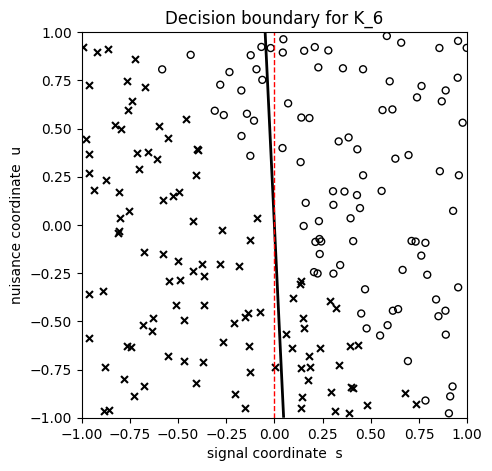} & \includegraphics[width=0.45\columnwidth]{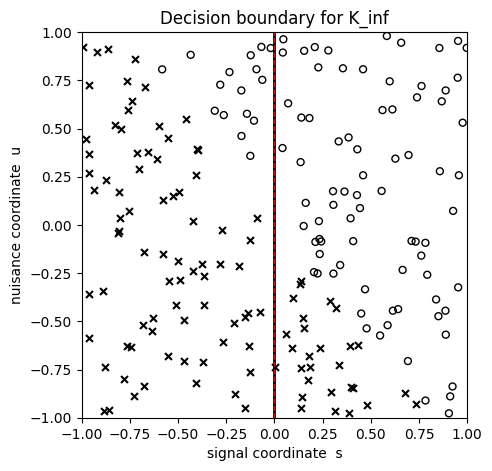}\tabularnewline
\end{tabular}

\caption{Evolution of the decision boundary under shorting dynamics. As $n$
increases, the kernel $K_{n}$ increasingly suppresses the nuisance
coordinate $u$, causing the decision boundary (solid lines) to rotate
from the tilted data-generating direction (at $n=0$) to the vertical
invariant limit (at $n=\infty$).}\label{fig:sd}
\end{figure}

\section{Concluding Remarks}

Shorting a fixed subspace is, of course, not a new problem. If one
knows a nuisance subspace $U\subset H$ and wishes to remove it entirely,
then projecting the feature map onto $U^{\perp}$ or using the classical
shorted operator already does the job. What the present work adds
is something different.

The residual-weighted update 
\[
R_{n+1}=R^{1/2}_{n}\left(I-T_{n+1}\right)R^{1/2}_{n}
\]
does not simply reproduce the static projection in a roundabout way.
It produces a path from the original geometry to the invariant one,
and the path has structure: the operators $\left(R_{n}\right)$, the
kernels $\left(K_{n}\right)$, and the Gram matrices $\left(G_{n}\right)$
all decrease monotonically, and each step comes with an explicit residual
identity. This allows one to see exactly what is being removed at
each stage and, in particular, to stop early when full invariance
is too harsh. In many learning problems the nuisance is real but not
completely irrelevant; having a gradual, controllable way to suppress
it can be useful.

A second point is that the dynamic can be aligned with a task. When
$T_{n+1}$ is chosen adaptively, the update removes the directions
that contribute most strongly to the task energy. In the rank-one
case this becomes a clean greedy principle: shorting along the top
eigenvector of $R^{1/2}_{n}BR^{1/2}_{n}$ achieves the largest possible
drop at that step. This type of feature elimination does not come
from the static theory, nor from the usual sample-level residualization
tricks.

Finally, the formulation keeps the operator, kernel, and sample pictures
in sync. The same recursion governs $\left(R_{n}\right)$, the induced
kernels, and the KRR predictors, and each level inherits the same
monotone and additive structure. This coherence is what allows the
method to be interpreted simultaneously in geometric, kernel-analytic,
and data-analytic terms.

The point, then, is not that shorting itself is new, or that one cannot
remove a known nuisance direction without the machinery developed
here. The point is that the dynamic exposes additional structure,
both geometric and algorithmic, that the static constructions do not
capture. Whether one uses the full limit or prefers to stop along
the way, the residual-weighted viewpoint offers a controlled and interpretable
way to modify a kernel while keeping careful track of what is being
lost and why.

\bibliographystyle{amsalpha}
\bibliography{ref}

\newcommand{\etalchar}[1]{$^{#1}$}
\providecommand{\bysame}{\leavevmode\hbox to3em{\hrulefill}\thinspace}
\providecommand{\MR}{\relax\ifhmode\unskip\space\fi MR }
\providecommand{\MRhref}[2]{%
  \href{http://www.ams.org/mathscinet-getitem?mr=#1}{#2}
}
\providecommand{\href}[2]{#2}
\begin{thebibliography}{DHP{\etalchar{+}}12}

\bibitem[ACS06]{MR2234254}
Jorge Antezana, Gustavo Corach, and Demetrio Stojanoff, \emph{Spectral shorted
  operators}, Integral Equations Operator Theory \textbf{55} (2006), no.~2,
  169--188. \MR{2234254}

\bibitem[AD69]{MR242573}
W.~N. Anderson, Jr. and R.~J. Duffin, \emph{Series and parallel addition of
  matrices}, J. Math. Anal. Appl. \textbf{26} (1969), 576--594. \MR{242573}

\bibitem[AMS07]{MR2306006}
Jorge Antezana, Pedro Massey, and Demetrio Stojanoff, \emph{Jensen's inequality
  for spectral order and submajorization}, J. Math. Anal. Appl. \textbf{331}
  (2007), no.~1, 297--307. \MR{2306006}

\bibitem[And71]{MR287970}
William~N. Anderson, Jr., \emph{Shorted operators}, SIAM J. Appl. Math.
  \textbf{20} (1971), 520--525. \MR{287970}

\bibitem[Aro50]{MR51437}
N.~Aronszajn, \emph{Theory of reproducing kernels}, Trans. Amer. Math. Soc.
  \textbf{68} (1950), 337--404. \MR{51437}

\bibitem[AT75]{MR356949}
W.~N. Anderson, Jr. and G.~E. Trapp, \emph{Shorted operators. {II}}, SIAM J.
  Appl. Math. \textbf{28} (1975), 60--71. \MR{356949}

\bibitem[Bha07]{MR2284176}
Rajendra Bhatia, \emph{Positive definite matrices}, Princeton Series in Applied
  Mathematics, Princeton University Press, Princeton, NJ, 2007. \MR{2284176}

\bibitem[CS02]{MR1864085}
Felipe Cucker and Steve Smale, \emph{On the mathematical foundations of
  learning}, Bull. Amer. Math. Soc. (N.S.) \textbf{39} (2002), no.~1, 1--49.
  \MR{1864085}

\bibitem[DHP{\etalchar{+}}12]{MR3388391}
Cynthia Dwork, Moritz Hardt, Toniann Pitassi, Omer Reingold, and Richard Zemel,
  \emph{Fairness through awareness}, Proceedings of the 3rd {I}nnovations in
  {T}heoretical {C}omputer {S}cience {C}onference, ACM, New York, 2012,
  pp.~214--226. \MR{3388391}

\bibitem[EE05]{MR2146819}
Herbert Egger and Heinz~W. Engl, \emph{Tikhonov regularization applied to the
  inverse problem of option pricing: convergence analysis and rates}, Inverse
  Problems \textbf{21} (2005), no.~3, 1027--1045. \MR{2146819}

\bibitem[EHN96]{MR1408680}
Heinz~W. Engl, Martin Hanke, and Andreas Neubauer, \emph{Regularization of
  inverse problems}, Mathematics and its Applications, vol. 375, Kluwer
  Academic Publishers Group, Dordrecht, 1996. \MR{1408680}

\bibitem[FW71]{MR293441}
P.~A. Fillmore and J.~P. Williams, \emph{On operator ranges}, Advances in Math.
  \textbf{7} (1971), 254--281. \MR{293441}

\bibitem[KC12]{Kamiran:2012aa}
Faisal Kamiran and Toon Calders, \emph{Data preprocessing techniques for
  classification without discrimination}, Knowledge and Information Systems
  \textbf{33} (2012), no.~1, 1--33.

\bibitem[Kre47a]{MR24574}
M.~Krein, \emph{The theory of self-adjoint extensions of semi-bounded
  {H}ermitian transformations and its applications. {I}}, Rec. Math. [Mat.
  Sbornik] N.S. \textbf{20(62)} (1947), 431--495. \MR{24574}

\bibitem[Kre47b]{MR24575}
M.~G. Kre\u{\i}n, \emph{The theory of self-adjoint extensions of semi-bounded
  {H}ermitian transformations and its applications. {II}}, Mat. Sbornik N.S.
  \textbf{21(63)} (1947), 365--404. \MR{24575}

\bibitem[LR86]{MR866966}
W.~E. Longstaff and Peter Rosenthal, \emph{On operator algebras and operator
  ranges}, Integral Equations Operator Theory \textbf{9} (1986), no.~6,
  820--830. \MR{866966}

\bibitem[Met97]{MR1465881}
Volker Metz, \emph{Shorted operators: an application in potential theory},
  Linear Algebra Appl. \textbf{264} (1997), 439--455. \MR{1465881}

\bibitem[MMP07]{MR2345997}
Alejandra Maestripieri and Francisco Mart\'{\i}nez~Per\'{\i}a, \emph{Schur
  complements in {K}rein spaces}, Integral Equations Operator Theory
  \textbf{59} (2007), no.~2, 207--221. \MR{2345997}

\bibitem[NRRR79]{MR531986}
E.~Nordgren, M.~Radjabalipour, H.~Radjavi, and P.~Rosenthal, \emph{On invariant
  operator ranges}, Trans. Amer. Math. Soc. \textbf{251} (1979), 389--398.
  \MR{531986}

\bibitem[Ped58]{MR2938971}
George Pedrick, \emph{T{HEORY} {OF} {REPRODUCING} {KERNELS} {FOR} {HILBERT}
  {SPACES} {OF} {VECTOR} {VALUED} {FUNCTIONS}}, ProQuest LLC, Ann Arbor, MI,
  1958, Thesis (Ph.D.)--University of Kansas. \MR{2938971}

\bibitem[PR16]{MR3526117}
V.~Paulsen and M.~Raghupathi, \emph{An introduction to the theory of
  reproducing kernel hilbert spaces}, Cambridge University Press, 2016.
  \MR{3526117}

\bibitem[RS72]{MR493419}
Michael Reed and Barry Simon, \emph{Methods of modern mathematical physics.
  {I}. {F}unctional analysis}, Academic Press, New York-London, 1972.
  \MR{493419}

\bibitem[SC08]{MR2450103}
Ingo Steinwart and Andreas Christmann, \emph{Support vector machines},
  Information Science and Statistics, Springer, New York, 2008. \MR{2450103}

\bibitem[Sza21]{MR4250453}
Franciszek~Hugon Szafraniec, \emph{Revitalising {P}edrick's approach to
  reproducing kernel {H}ilbert spaces}, Complex Anal. Oper. Theory \textbf{15}
  (2021), no.~4, Paper No. 66, 12. \MR{4250453}

\end{thebibliography}

\end{document}